\documentclass[a4paper,10pt]{amsart}

\usepackage{graphicx}
\usepackage{dirtytalk}
\usepackage{amsthm}
\usepackage{amssymb}
\usepackage{amsmath}
\usepackage{accents}
\usepackage{hyperref}
\usepackage{xcolor}
\usepackage{enumerate}
\usepackage{listings}
\usepackage{complexity}
\usepackage{blkarray}
\usepackage{braket}
\usepackage{lmodern}
\usepackage[english]{babel}
\usepackage{enumitem}
\usepackage{float}

\usepackage[font=small,labelfont=bf]{caption}
\usepackage[font=small,labelfont=normalfont,labelformat=simple]{subcaption}
\graphicspath{{figs/}}

\newtheorem{theorem}{Theorem}
\newtheorem{definition}{Definition}
\newtheorem{lemma}[theorem]{Lemma}

\newtheorem{observation}{Observation}

\newtheorem{conjecture}{Conjecture}
\newtheorem{problem}{Problem}

\newcommand{\zig}{\mathrm{zig}}

\author
{
Raphael Steiner 
}
\thanks{Department of Computer Science, Institute of Theoretical Computer Science, ETH Z\"{u}rich, Switzerland.  \texttt{raphaelmario.steiner@inf.ethz.ch}. Research funded by SNSF Ambizione grant No. 216071.}

\date{\today}

\title{Hadwiger's conjecture and topological bounds}

\begin{document}
\maketitle

\begin{abstract}
The Odd Hadwiger's conjecture, formulated by Gerards and Seymour in 1995, is a substantial strengthening of Hadwiger's famous coloring conjecture from 1943. We investigate whether the hierarchy of topological lower bounds on the chromatic number, introduced by Matou\v{s}ek and Ziegler (2003) and refined recently by Daneshpajouh and Meunier (2023), forms a potential avenue to a disproof of Hadwiger's conjecture or its odd-minor variant. 

\noindent In this direction, we prove that, in a very general sense, every graph $G$ that admits a topological lower bound of $t$ on its chromatic number, contains $K_{\lfloor t/2\rfloor +1}$ as an odd-minor. This solves a problem posed by Simonyi and Zsb\'{a}n [European Journal of Combinatorics, 31(8), 2110--2119 (2010)].

\noindent We also prove that if for a graph $G$ the Dol'nikov-K\v{r}\'{i}\v{z} lower bound on the chromatic number (one of the lower bounds in the aforementioned hierarchy) attains a value of at least $t$, then $G$ contains $K_t$ as a minor. 

\noindent Finally, extending results by Simonyi and Zsb\'{a}n, we show that the Odd Hadwiger's conjecture holds for Schrijver and Kneser graphs for any choice of the parameters. The latter are canonical examples of graphs for which topological lower bounds on the chromatic number are tight. 
\end{abstract}

\section{Introduction}

Given two finite graphs $G$ and $H$, we say that $G$ contains $H$ as a \emph{minor} if there exists a sequence of vertex-deletions, edge-deletions and edge-contractions (in arbitrary order) that transforms $G$ into a graph that is isomorphic to $H$. 
Hadwiger's conjecture, formulated more than 80 years ago by Hadwiger~\cite{hadwiger}, is the major open problem in the theory of graph coloring, and arguably one of the most challenging open problems in all of combinatorics. It suggests a deep connection between the chromatic number of a graph and the containment of clique-minors.
\begin{conjecture}[Hadwiger 1943~\cite{hadwiger}]\label{con:hadwiger}
Let $t$ be a positive integer. Every graph $G$ satisfying $\chi(G)\ge t$ contains $K_t$ as a minor.
\end{conjecture}
Hadwiger's conjecture is known to hold for $t\le 6$, but remains widely open for any value $t \ge 7$. It was proven for $t\le 4$ by Hadwiger himself~\cite{hadwiger} and independently by Dirac~\cite{dirac}. The case $t=5$ was proven to be equivalent to the $4$-Color-Theorem by Wagner~\cite{wagner}, and also the case $t=6$ was reduced to the $4$-Color-Theorem in a landmark paper by Robertson, Seymour and Thomas~\cite{robertson} in 1993. Recently, there has been some exciting progress towards the linear weakening of Hadwiger's conjecture, which states that there exists some absolute constant $C>0$ such that every graph $G$ with $\chi(G)\ge Ct$ contains $K_t$ as a minor~\cite{norinpostlesong,del}. The current state of the art-result coming from these works is a theorem by Delcourt and Postle~\cite{del}, stating that for some absolute constant $C>0$ and $t$ large enough every graph of chromatic number at least $Ct \log\log t$ contains $K_t$ as a minor. For more background on Hadwiger's conjecture, we refer to Seymour's 2016-survey~\cite{seymour} and a very recent update on the developments by Norin~\cite{norinsurvey}. 

A substantial qualitative strengthening of Hadwiger's conjecture was proposed in 1995 by Gerards and Seymour~\cite{gerards}. It states that, in any graph of chromatic number at least $t$, not only should it be possible to find $K_t$ as a minor, but in fact one should be able to find $K_t$ as a so-called \emph{odd-minor}. Given two graphs $G$ and $H$, we say that $G$ contains $H$ as an odd-minor, if a graph isomorphic to $H$ can be obtained from a subgraph $G'\subseteq G$ by selecting an edge-cut in $G'$ and contracting all the edges in this cut simultaneously. It is clear by definition that whenever a graph $G$ contains a graph $H$ as an odd-minor, it also contains $H$ as a minor. However, odd-minors are significantly more restricted than minors, for instance they preserve the parity of cycle-lengths. A good example to illustrate this difference are the complete bipartite graphs: While they can contain huge cliques as ordinary minors, none of them contains even $K_3$ as an odd-minor, since the containment of $K_3$ as an odd-minor necessitates the existence of an odd cycle.

\begin{conjecture}[Gerards and Seymour 1995~\cite{gerards}]\label{con:odd}
Let $t$ be a positive integer. Every graph $G$ satisfying $\chi(G)\ge t$ contains $K_t$ as an odd-minor.
\end{conjecture}

Given that Conjecture~\ref{con:odd} is a substantial strengthening of Hadwiger's tremendously challenging conjecture, it is not surprising that it has resisted all attempts of resolution thus far. The cases $t \le 4$ of the conjecture were proven by Catlin~\cite{catlin}, and a proof of the case $t=5$ (a strengthening of the $4$-Color-Theorem!) was announced by Guenin in 2004, cf.~\cite{seymour}, but still has not been made public. While the precise form of the conjecture remains widely open for $t \ge 6$, as for Hadwiger's conjecture, there has been progress regarding asymptotic results~\cite{geelen,norinsongodd,postleodd,steiner}, with the current state of the art being a bound of $O(t\log\log t)$ on the chromatic number of odd $K_t$-minor-free graphs from~\cite{steiner}. 
\medskip
\paragraph{\textbf{Topological lower bounds.}} If, as many researchers in the community seem to believe, the precise or even asymptotic form of Hadwiger's conjecture or its odd-minor variant is false, then one should put efforts into constructing counterexamples. In order to do so, it would be necessary to construct $K_t$-minor-free graphs of high chromatic number. But chromatic number is a surprisingly complex parameter, and we do not really understand in depth what makes the chromatic number of a graph large---in particular, approximating the chromatic number of an $n$-vertex graph to within a factor better than $n^{1-\varepsilon}$ is computationally intractable~\cite{hardness}. This suggests that one will have to rely on one of the few known lower bounds on the chromatic number to construct such graphs. One of the simplest such lower bounds is the clique number $\omega(G)$, but since a clique subgraph directly yields a clique minor of the same order, it seems badly suited for the problem. A much stronger and more general lower bound is the \emph{fractional chromatic number} $\chi_f(G)$, that is obtained by a linear programming relaxation of an integer program for the chromatic number. In fact, for many graphs of large chromatic number that one typically encounters, including random graphs, $\chi_f(G)$ will also be large and close to $\chi(G)$. However, the hopes of disproving Hadwiger's or the Odd Hadwiger's conjecture using fractional chromatic number as a lower bound are dampened by results of Reed and Seymour~\cite{reedseymour} and Kawarabayashi and Reed~\cite{kawarabayashireed}, showing that any graph $G$ satisfying $\chi_f(G)> t-1$ contains a rather big clique, namely $K_{\lfloor (t+1)/2\rfloor}$, as an odd-minor. If we cannot rely on fractional chromatic number as a lower bound to disprove (Odd) Hadwiger's conjecture, maybe we should consider lower bounds that work orthogonally to the fractional chromatic number, in the sense that they can be large on graphs with small fractional chromatic number. The most well-explored such bounds are the \emph{topological lower bounds} on the chromatic number. The probably most famous usage of topology to determine the chromatic number of a graph is the proof of Kneser's conjecture by Lov\'{a}sz~\cite{lovasz}, showing that the chromatic number of the Kneser graph\footnote{For a definition of Kneser graphs, see Definition~\ref{def:kneserschrijver}.} $K(n,k)$ with $n \ge 2k$ equals $n-2k+2$. This is much larger than the fractional chromatic number of these graphs, which equals $n/k$. 

In 2003, Matou\v{s}ek and Ziegler~\cite{matousekziegler} generalized the method for lower-bounding the chromatic number of Kneser graphs employed by Lov\'{a}sz to arbitrary graphs $G$ via topological parameters based on the so-called \emph{box complexes} $\mathsf{B}(G)$ and $\mathsf{B}_0(G)$. These are certain abstract simplicial complexes associated with $G$, as follows.

\begin{definition}[Box complexes, cf.~\cite{matousekziegler,danesh}] Let $G$ be a graph. The vertex-set of both $B(G)$ and $B_0(G)$ is $V(G)\times \{1,2\}$. A subset $X\subseteq V(G)\times \{1,2\}$ forms a face of $\mathsf{B}_0(G)$ if and only if the two projections $X_1:=\{x \in V(G)|(x,1) \in X\}$ and $X_2:=\{x \in V(G)|(x,2) \in X\}$ are complete to each other in $G$, i.e., form the two sides of a complete bipartite subgraph of $G$ (the sides are allowed to be empty). Similarly, a subset $X\subseteq V(G)\times \{1,2\}$ forms a face of $\mathsf{B}(G)$ if it forms a face of $\mathsf{B}_0(G)$ and additionally both the projections $X_1, X_2$ of $X$ have at least one common neighbor in $G$.
\end{definition}

With every abstract simplicial complex $K=(V,\mathcal{F})$, one can naturally associate a topological space, called its \emph{geometric realization}. One way to define this space is to consider a geometric representation of $K$ in $\mathbb{R}^d$ for some sufficiently large\footnote{Such a representation can be obtained, for instance, by mapping $V$ to the set of unit vectors $\{e_v|v \in V\}$ in $\mathbb{R}^V$ and considering the collection of simplices that is obtained by assigning to each face $F \in \mathcal{F}$ the simplex $\mathrm{conv}\{e_v|v \in F\}$.} $d$ and to take the subspace-topology induced on this geometric representation by $\mathbb{R}^d$ with its natural topology. This allows to define various topological parameters on simplical complexes, including the box complexes $\mathsf{B}(G), \mathsf{B}_0(G)$. Via the Borsuk-Ulam theorem on coverings of high-dimensional spheres and the relation of graph homomorphisms to so-called $\mathbb{Z}_2$-maps between their associated box-complexes, several of those parameters can be used to give lower bounds on the chromatic number $\chi(G)$.

\begin{figure}[H]
    \centering
    \includegraphics[scale=0.55]{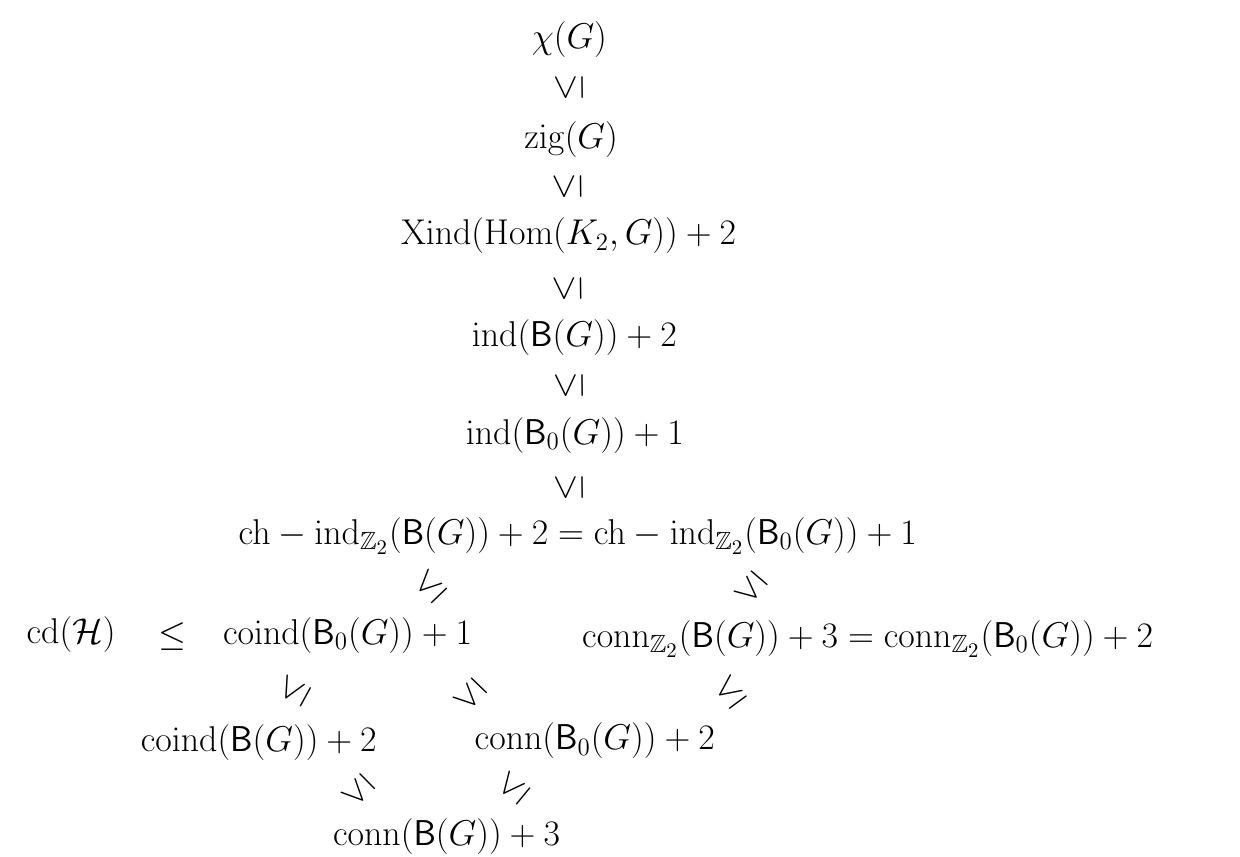}
    \caption{Illustration of the topological lower bounds on the chromatic number of a graph and their relations.}\label{fig:hierarchy}
\end{figure}

In the well-written and comprehensive surveys on this topic by Matou\v{s}ek and Ziegler~\cite{matousekziegler} from 2003 and by Daneshpajou and Meunier~\cite{danesh} from 2023, these parameters are compared to each other and put into a hierarchy. In Figure~\ref{fig:hierarchy} that is emulated from the paper by Daneshpajou and Meunier, we display the currently known topological lower bounds on the chromatic number and their relations. 
\medskip
\paragraph{\textbf{Combinatorial parameters.}} Since the focus of the paper at hand is combinatorial, and since it will not be necessary to understand our results or proofs, we refrain from giving the detailed definitions of all the parameters displayed in the figure, and instead focus on the two purely combinatorial ones, namely $\zig(G)$ and $\mathrm{cd}(\mathcal{H})$. Detailed explanations of all parameters and background can be found in~\cite{danesh}. 

\begin{definition}[Zig-Zag-number, cf.~\cite{danesh}, Section~4.1]
Let $c:V(G)\rightarrow \mathbb{N}$ be a proper coloring of a graph $G$. A \emph{zigzag} of length $k$ in the properly colored graph $(G,c)$ is defined as a sequence of vertices $z_1,\ldots,z_k\in V(G)$ such that $c(z_1)<\cdots<c(z_k)$ and such that the sets of vertices $\{z_i|i\text{ odd}\}$ and $\{z_i |i \text{ even}\}$ are complete to each other in $G$, i.e., form the two sides of a complete bipartite subgraph. 

\bigskip
The \emph{zig-zag-number} $\zig(G)$ of a graph is then defined as the minimum, over all proper colorings $c:V(G)\rightarrow \mathbb{N}$ of $G$, of the maximum length of a zigzag in $(G,c)$. 
\end{definition}
For most of the parameters in Figure~\ref{fig:hierarchy} the proof that they provide a lower bound on $\chi(G)$ relies on topology, and in most cases on the Borsuk-Ulam theorem. However, one does not need topology to see that $\zig(G)\le \chi(G)$ for every graph $G$, since the latter follows directly by definition of a zigzag. Furthermore, note that $\zig(G)$ is the strongest of all the lower bounds in Figure~\ref{fig:hierarchy}. Thus, $\zig(G)$ provides a combinatorial tool to simultaneously bound all the known topological parameters that provide lower bounds on $\chi(G)$. 
\begin{definition}[$2$-colorability defect, Kneser representation, cf.~\cite{matousekziegler}, page 8--9 and~\cite{danesh}, Section~4.1]
Let $\mathcal{H}=(V,E)$ be a hypergraph. The \emph{Kneser graph} $\mathrm{KG}(\mathcal{H})$ induced by $\mathcal{H}$ is the graph with vertex-set $E$ in which two hyperedges $e_1, e_2$ are adjacent if and only if $e_1 \cap e_2=\emptyset$. If $G$ is a graph, we say that a hypergraph $\mathcal{H}$ is a \emph{Kneser-representation} of $G$ if $\mathrm{KG}(\mathcal{H})$ is isomorphic to $G$. 

\bigskip
The $2$-colorability defect of a hypergraph $\mathcal{H}=(V,E)$ is defined as\footnote{Recall that a hypergraph is said to be $2$-colorable if it is possible to assign one of two colors to each vertex such that no hyperedge is fully contained in one color class.}
$$\mathrm{cd}(\mathcal{H}):=\min\{|U| : U \subseteq V\text{ such that }(V\setminus U, \{e\in E: e\cap U=\emptyset\})\text{ is }2\text{-colorable}\}.$$
\end{definition}
It is worth mentioning and not hard to prove that every graph admits at least one (and in fact, many) Kneser-representations. Therefore, the following result, a wide-ranging generalization of Lov\'{a}sz' lower bound on the chromatic number of Kneser graphs due to Dol'nikov~\cite{dolnikov} and K\v{r}\'{i}\v{z}~\cite{kriz}, can be used as a lower bound for the chromatic number of any graph $G$. 

\begin{theorem}[Dol'nikov~\cite{dolnikov}, K\v{r}\'{i}\v{z}~\cite{kriz}]
Let $G$ be a graph with a Kneser-representation $\mathcal{H}$. Then $\chi(G)\ge \mathrm{cd}(\mathcal{H})$. 
\end{theorem}

\medskip
\paragraph{\textbf{Our results.}} In this paper, we investigate how the hierarchy of topological lower bounds on the chromatic number interacts with the presence of clique minors, and thus evaluate the potential these lower bounds provide for attacking Hadwiger's conjecture or its odd-minor variant. The investigation of this topic was initiated in previous work of Simonyi and Zsb\'{a}n~\cite{simonyi} and even earlier by Simonyi and Tardos~\cite{simonyitardos}, where they investigate the relaxations of Hadwiger's conjecture and its odd-minor variant obtained by putting one of the topological lower bounds on the chromatic number in place of $\chi(G)$. 

Regarding Hadwiger's conjecture, Simonyi and Tardos (see Corollary~25 in~\cite{simonyitardos}) observed that known results from the literature, in particular the so-called \emph{$K_{\ell,m}$-theorem} and the stronger \emph{Zig-Zag-theorem} (relating to $\zig(G)$ defined above), imply that every graph for which one of the topological lower bounds on the chromatic number attains a value of at least $t$, contains $K_{\lfloor t/2\rfloor, \lceil t/2\rceil}$ as a subgraph and therefore $K_{\lfloor t/2\rfloor+1}$ as a minor. Due to this observation, Simonyi and Zsb\'{a}n~\cite{simonyi} focused on the stronger Odd Hadwiger's conjecture and stated that ``Our main concern here is that the Odd Hadwiger's conjecture looks much more difficult in this respect''. Indeed, while the presence of a large clique minor can be directly deduced from the presence of a large balanced complete bipartite subgraph, the latter does not at all imply the existence of a large odd clique minor---as mentioned before, bipartite graphs are odd-$K_3$-minor-free. In fact, for the right choice of parameters, graphs like the Kneser graphs $K(n,k)$, for which the topological lower bounds on the chromatic number are tight, can exhibit arbitrarily large odd girth, meaning that it will not be possible to find a bounded-size subgraph that provides a large odd clique minor. This led Simonyi and Zsb\'{a}n to explicitely pose the following open problem in~\cite{simonyi}, which remained unsolved thus far:

\begin{problem}[Problem in~Section~5 of~\cite{simonyi}]
Does there exist an absolute constant $c>0$ for which the following is true? If a graph $G$ satisfies $\zig(G)\ge t$, then $G$ contains an odd complete minor on at least $ct$ vertices.
\end{problem}

As the main result of this paper, we solve the above problem in the affirmative, with the constant $c=\frac{1}{2}$:

\begin{theorem}\label{thm:main}
Let $t$ be a positive integer. If a graph $G$ satisfies $\zig(G)\ge t$, then $G$ contains $K_{\lfloor t/2\rfloor+1}$ as an odd-minor.
\end{theorem}

As one can see from Figure~\ref{fig:hierarchy}, if any one of the known topological lower bounds for $\chi(G)$ attains a value of at least $t$, also $\zig(G)$ is of value at least $t$, and thus the statement of Theorem~\ref{thm:main} yields the presence of large odd clique minors in all graphs for which at least one of the topological parameters displayed in the figure gets large.

In the results mentioned so far, both for the topological relaxation of Hadwiger's and of the Odd Hadwiger's conjecture, we can only find a clique minor of half the size that would be guaranteed if the conjectures were true. Motivated by this gap, in our second main result, we show that the precise statement of Hadwiger's conjecture can be recovered when chromatic number is replaced by one of the lower bounds in Figure~\ref{fig:hierarchy}, namely $\mathrm{cd}(\mathcal{H})$:

\begin{theorem}\label{thm:dolnikov}
Let $t$ be a positive integer. If a graph $G$ admits a Kneser representation $\mathcal{H}$ such that $\mathrm{cd}(\mathcal{H})\ge t$, then $G$ contains $K_t$ as a minor.
\end{theorem}

In particular, this shows that Hadwiger's conjecture is true for any graph for which the Dol'nikov-K\v{r}\'{i}\v{z} bound provides a tight estimate of its chromatic number (such as Kneser graphs, for example).

It would be nice to also get a similar tight result for other parameters in Figure~\ref{fig:hierarchy}, or to strengthen Theorem~\ref{thm:dolnikov} to odd-minors. While we did not succeed at this task in general, we could at least prove that the precise version of Odd Hadwiger's conjecture does hold for the most prominent examples of graphs for which the topological lower bounds are tight, namely for all Kneser and all Schrijver graphs.

\begin{definition}[Kneser graph, Schrijver graph, cf.~\cite{schrijver}]\label{def:kneserschrijver}
Let $n, k$ be positive integers with $n \ge 2k$. The Kneser graph $K(n,k)$ is the graph with vertex-set $\binom{[n]}{k}$ in which two $k$-subsets of $[n]=\{1,\ldots,n\}$ are made adjacent iff they are disjoint. The Schrijver graph $S(n,k)$ is the subgraph of $K(n,k)$ induced by all those $k$-subsets of $[n]$ that are \emph{cyclically stable}, i.e. that do not contain two consecutive elements in the natural cyclical ordering of $[n]$. 
\end{definition} 
Following Lov\'{a}sz proof of Kneser's conjecture that $\chi(K(n,k))=n-2k+2$, Schrijver~\cite{schrijver} proved that the same tight estimate holds for all Schrijver graphs, i.e., $\chi(S(n,k))=n-2k+2$ and that moreover, $S(n,k)$ is minimal with this property, in that the chromatic number drops if any vertex is removed. The following result proves the Odd Hadwiger's conjecture for all Kneser and Schrijver graphs. 
\begin{theorem}\label{thm:schrijver}
    For every choice of positive integers $n,k$ such that $n \ge 2k$, the Schrijver graph $S(n,k)$ (and thus also the Kneser graph $K(n,k)$) contains $K_{n-2k+2}$ as an odd-minor. 
\end{theorem}
We remark that under the restrictive assumption that $n-2k+2=t$ is fixed as a constant while $n$ and $k$ are sufficiently large as a function of $t$, the result in Theorem~\ref{thm:schrijver} was already proved by Simonyi and Zsb\'{a}n~\cite{simonyi}. Our result is more general in that no conditions on $n$ and $k$ are required for Theorem~\ref{thm:schrijver} to hold. 

\section{Proofs of Theorems~\ref{thm:main} and~\ref{thm:dolnikov}}

In this section, we present the proofs of Theorem~\ref{thm:main} and Theorem~\ref{thm:dolnikov}. The proof of Theorem~\ref{thm:main} builds on the following key lemma from~\cite{steiner}. The statement given here is slightly stronger than the statement of Lemma~4 in~\cite{steiner}, but looking at the proof of Lemma~4 in~\cite{steiner} one readily checks that it does in fact yield Lemma~\ref{lemma:decompose} as stated here.

\begin{lemma}[cf.~Lemma 4 in~\cite{steiner}]\label{lemma:decompose}
Let $G$ be a graph. There exists $n \in \mathbb{N}$ and a partition of $V(G)$ into $n$ non-empty sets $X_1,\ldots,X_n$ such that the following holds:
\begin{enumerate}
    \item For every $1\le i \le n$, the graph $G[X_i]$ is bipartite and connected.
    \item For every $1\le i<j\le n$, and every vertex $v \in X_j$, the following holds. Either $v$ has no neighbors in $X_i$, or it has two distinct neighbors $u_1, u_2$ where $u_1$ and $u_2$ lie on different sides of the bipartition of $G[X_i]$.
\end{enumerate}
\end{lemma}

To present the proof of Theorem~\ref{thm:main}, it will be convenient to work with the notion of \emph{odd $K_k$-expansions} from~\cite{geelen}, defined as follows.

\begin{definition}
A $K_k$-expansion is a graph $F$ that consists of the disjoint union of $k$ trees $T_1,\ldots,T_k$, together with exactly one edge joining every pair of two distinct trees. We say that $F$ is an \emph{odd} $K_k$-expansion, if there exists a color-assignment $c_{\text{tree}}:V(T_1)\cup \cdots \cup V(T_k)\rightarrow \{1,2\}$ such that for every $i=1,\ldots,k$, the restriction of $c_{\text{tree}}$ to $V(T_i)$ forms a proper coloring of the tree $T_i$, while every edge connecting two distinct trees is monochromatic with respect to $c_{\text{tree}}$. 
\end{definition}
The following fact is folklore, see also~\cite{geelen}:
\begin{observation}
    If $G$ is a graph and $k$ a positive integer, then $G$ contains $K_k$ as an odd-minor if and only if there exists a subgraph of $G$ that is an odd $K_k$-expansion. 
\end{observation}

We are now ready for the proof of Theorem~\ref{thm:main}. 

\begin{proof}[Proof of Theorem~\ref{thm:main}]
Since $\lfloor t/2\rfloor+1=\lfloor (t-1)/2\rfloor+1$ if $t$ is an odd number, we may restrict ourselves to proving the assertion of Theorem~\ref{thm:main} for the case when $t$ is an even number. Let $k\ge 2$ be an integer such that $t=2(k-1)$. Our goal is then to prove that $G$ contains $K_{\lfloor t/2\rfloor+1}=K_k$ as an odd-minor. 

We start by constructing a proper coloring $c:V(G)\rightarrow \mathbb{N}$ of $G$ as follows. We take a partition $\mathcal{X}=(X_i)_{i=1}^{n}$ of the vertex-set of $G$ as given by Lemma~\ref{lemma:decompose}, and let $\eta:V(G)\rightarrow \{1,\ldots,n\}$ be the mapping that assigns to any vertex $x \in V(G)$ the unique index $i=:\eta(x)$ satisfying $x \in X_{i}$. 
Next, we let $A_i, B_i$ denote the two sides of the unique bipartition of $G[X_i]$ for every $i$. We then define  $$c(x):=\begin{cases}2\eta(x)-1, & \text{ if }x \in A_{\eta(x)}, \\ 2\eta(x), & \text{ if }x \in B_{\eta(x)} \end{cases}$$ for every $x\in V(G)$. Clearly $c$ is a proper coloring of $G$, and thus $\zig(G)\ge t=2k-2$ implies the existence of a sequence of vertices $z_1,z_2,\ldots,z_{2k-2}$ in $G$ such that $c(z_1)<c(z_2)<\ldots<c(z_{2k-2})$ and such that the edges between the two vertex-sets $Z_1:=\{z_i|i \text{ odd}\}$ and $Z_2:=\{z_i|i \text{ even}\}$ form a complete bipartite subgraph $B$ of $G$, which is isomorphic to $K_{k-1,k-1}$.

Let $M:=\{z_iz_j|1\le i<j\le 2k-2 \text{ and } \eta(z_i)=\eta(z_j)\}$. 
Observe that by definition of the coloring $c$, if $z_iz_j\in M$, then we have $|c(z_i)-c(z_j)| \le 1$. Using the strict ordering $c(z_1)<\cdots<c(z_{2k-2})$, it follows that $j=i+1$. This implies that $z_iz_j$ forms an edge of $B$. It is furthermore clear from the definition of $c$ and $c(z_1)<\cdots<c(z_{2k-2})$ that any two distinct pairs in $M$ must be disjoint. Thus, $M$ forms a matching in $B$. In the following, for each $i=1,\ldots,n$ pick and fix a spanning tree $T_i$ of the connected graph $G[X_i]$. The following claim will be used repeatedly in the arguments below.

\medskip

\paragraph{\textbf{Claim~($\ast$).} Let $I \subseteq [n]$ and let $f:\bigcup_{i \in I}{V(T_i)}\rightarrow \{1,2\}$ be any assignment that, restricted to each tree $T_i$ with $i \in I$, forms a proper coloring. If for some distinct indices $i, j \in I$ there exists an edge in $G$ connecting $T_i$ and $T_j$, then there also exists a monochromatic edge (with respect to $f$) connecting $T_i$ and $T_j$. }
\begin{proof}
W.l.o.g. assume $i<j$ and let $ab\in E(G)$ be such that $a \in V(T_i), b \in V(T_j)$. By choice of the partition $\mathcal{X}$, we know that $b$ has neighbors on both sides of the unique bipartition of the connected bipartite subgraph $G[X_{i}]$, which is the same as the unique bipartition of the spanning tree $T_{i}$. Since $f$ properly colors each of the trees $T_1,\ldots,T_n$, this implies that $b$ has neighbors of both colors $1$ and $2$ in $T_{i}$, and so in particular a neighbor of the same color as itself.
\end{proof}

We split the rest of the proof into two cases, depending on the size of $M$.
\medskip
\paragraph{\textbf{Case 1.} $|M|\le k-2$.} Let $M'\supseteq M$ be a perfect matching of the complete bipartite graph $B$ that extends $M'$. Let us label the edges in $M'$ as $M'=\{x_1y_1,\ldots,x_{k-1}y_{k-1}\}$ such that $x_1,\ldots,x_{k-1}\in Z_1$, $y_1,\ldots,y_{k-1}\in Z_2$ as well as $M=\{x_1y_1,\ldots,x_ry_r\}$ for some $0\le r \le k-2$. For each index $j \in \{1,\ldots,k\}$, we define a tree $T_j' \subseteq G$ as follows. If $j \le r$, then we define $T_j':=T_{\eta(x_j)}=T_{\eta(y_j)}$. If $r+1\le j\le k-2$, then define $T_j'$ as the tree obtained from the disjoint union of $T_{\eta(x_j)}$ and $T_{\eta(y_j)}$ by adding the edge $x_jy_j$. Finally, let $T_{k-1}':=T_{\eta(x_{k-1})}$ and $T_k':=T_{\eta(y_{k-1})}$. Let us consider a mapping $c_{\text{tree}}:\bigcup_{j=1}^{k}{V(T_j')} \rightarrow \{1,2\}$ that is obtained by piecing together proper colorings of the trees $T_1',\ldots,T_k'$. We claim that for every $1 \le j_1 <j_2 \le k$, there exists an edge in $G$ connecting $T_{j_1}'$ and $T_{j_2}'$ that is monochromatic with respect to $c_{\text{tree}}$. Indeed, by definition there exists at least one edge in $G$ that connects $T_{j_1}'$ and $T_{j_2}'$, namely the edge $x_{j_1}y_{j_2}$ if $j_2\le k-2$, the edge $y_{j_1}x_{k-1}$ if $j_2=k-1$, and the edge $x_{j_1}y_{k-1}$ if $j_2=k$. Since the vertex-set of both trees $T_{j_1}'$ and $T_{j_2}'$ are disjoint unions of the vertex-sets of some of the trees $T_1,\ldots,T_n$, by Claim~($\ast$) the existence of at least one edge between $T_{j_1}'$ and $T_{j_2}'$ already implies the existence of such an edge which is monochromatic w.r.t.~$c_{\text{tree}}$. This shows that $T_1',\ldots,T_k'$ indeed form an odd $K_k$-expansion in $G$, and this concludes the proof in this first case.
\medskip
\paragraph{\textbf{Case 2.} $|M|=k-1$.} Label the edges in $M$ as $x_1y_1,\ldots,x_{k-1}y_{k-1}$ such that  $x_1,\ldots,x_{k-1} \in Z_1$, $y_1,\ldots,y_{k-1}\in Z_2$ and $\eta(x_j)=\eta(y_j)<\eta(x_{k-1})=\eta(y_{k-1})$ for all $1\le j\le k-2$. Let $T_j':=T_{\eta(x_j)}=T_{\eta(y_j)}$ for every $1\le j \le k-2$, and let $T_{k-1}'$ and $T_k'$ be defined as the single-vertex trees consisting of $x_{k-1}$ and $y_{k-1}$, respectively. Let $c_{\text{tree}}:\bigcup_{j=1}^{k}{V(T_j')} \rightarrow \{1,2\}$ be obtained by piecing together proper colorings of the trees $T_1',\ldots,T_{k-2}'$ (which by definition are clearly vertex-disjoint) and assigning color $1$ to both $x_{k-1}$ and $y_{k-1}$. We now claim that between any two of the trees $T_1',\ldots,T_{k}'$, there exists at least one connecting edge in $G$ that is monochromatic w.r.t.~$c_{\text{tree}}$. For any $1\le j_1<j_2 \le k-2$, it follows directly by definition that $x_{j_1}y_{j_2}$ forms a connecting edge between $T_{j_1}=T_{\eta(x_{j_1})}$ and $T_{j_2}=T_{\eta(y_{j_2})}$, and thus by Claim~($\ast$) also a monochromatic connecting edge w.r.t.~$c_{\text{tree}}$ must exist. Furthermore, it follows directly by definition that $x_{k-1}y_{k-1}$ is a monochromatic edge connecting the trees $T_{k-1}'$ and $T_{k-2}'$. Finally, for every $1 \le j \le k-2$ we have that $\eta(x_j)=\eta(y_j)<\eta(x_{k-1})=\eta(y_{k-1})$, and both $x_{k-1}$ and $y_{k-1}$ have at least one neighbor in $T_j'=T_{\eta(x_j)}$, namely $y_j$ and $x_j$, respectively. These facts and the choice of $\mathcal{X}$ imply that both $x_{k-1}$ and $y_{k-1}$ have neighbors on both sides of the bipartition of $G[X_{\eta(x_j)}]$, which is the same as the bipartition of $T_j'$. Hence, both of them are connected to a vertex in $T_j'$ via a monochromatic edge. This establishes that $T_1',\ldots,T_k'$ form an odd $K_k$-expansion in $G$, concluding the proof also in the second case.
\end{proof}

Next, we present the compact proof of Theorem~\ref{thm:dolnikov}. 

\begin{proof}[Proof of Theorem~\ref{thm:dolnikov}]
Let $G$ be a graph that admits a Kneser-representation $\mathcal{H}=(V,E)$ such that $\mathrm{cd}(\mathcal{H})\ge t$. W.l.o.g. we may assume $G=\text{KG}(\mathcal{H})$. Let $U\subseteq V$ be a smallest set such that the restricted hypergraph $\mathcal{H}-U:=(V\setminus U, \{e\in E: e\cap U=\emptyset\})$ is $2$-colorable. Then $|U|=\mathrm{cd}(\mathcal{H})\ge t$. Finally, let $X_1, X_2$ denote the two color classes of a $2$-coloring of the hypergraph $\mathcal{H}-U$. 
\medskip
\paragraph{\textbf{Claim 1.} For every $u \in U$ and every $i \in \{1,2\}$ there exists a hyperedge $e \in E$ such that $e \subseteq X_i \cup \{u\}$ and $u \in e$.}
\begin{proof}
    Suppose that for some $u \in U$ and $i \in \{1,2\}$ no hyperedge $e$ with the stated properties exists. Then there exists no hyperedge of $\mathcal{H}$ that is fully contained in the set $X_i \cup \{u\}$, and thus the two sets $X_i \cup \{u\}$ and $X_{2-i}$ form the color classes of a proper coloring of the hypergraph $(V\setminus (U\setminus\{u\}), \{e\in E: e\cap (U\setminus \{u\})=\emptyset\})$. This contradicts our minimality assumption on $U$, which concludes the proof.
\end{proof}
Let $U_1:=\{u \in U|\{u\}\in E\}$ and $U_2:=U\setminus U_1$. By the Claim~1, for every $u \in U_2$ and every $i \in \{1,2\}$ there exists some edge $e(u,i) \in E$ such that $u \in e(u,i) \subseteq X_i \cup \{u\}$ and, since $\{u\}\notin E$, such that $e(u,i)\cap X_i \neq \emptyset$. These properties imply in particular that all elements of the multi-set $$S:=\{\{u\}\{u'\}|u \neq u' \in U_1\}\cup \{e(u,i)|u \in U_2, i \in \{1,2\}\}$$ are pairwise distinct, and so the latter is a set. 

By definition of these hyperedges, for every choice of elements $u,u'\in U$ and $i,j \in \{1,2\}$ such that $u_1 \neq u_2$ and $i \neq j$, we have $$e(u,i)\cap e(u',j)\subseteq (X_i\cup \{u\})\cap (X_j\cup \{u'\})=\emptyset,$$ and hence $e(u,i)$ and $e(u',j)$ are adjacent in $G$. Furthermore, by definition of $U_1$, for every $u_1 \in U_1$, $u_2 \in U_2$ and $i \in \{1,2\}$, we have $\{u_1\}\cap e(u_2,i)\subseteq U_1 \cap (U_2 \cup X_i)=\emptyset$, and hence $\{u_1\}$ and $e(u_2,i)$ are adjacent in $G$. Finally, the singletons $\{u\}, u \in U_1$ are clearly pairwise adjacent in $G$. 

Altogether, this implies that the graph $H$, defined via $V(H):=S$ and \begin{align*} E(H):&=\{\{u\}\{u'\}|u \neq u' \in U_1\} \\ &\cup \, \{\{u_1\}e(u_2,i)|u_1 \in U_1, u_2 \in U_2, i \in \{1,2\}\} \\ &\cup \, \{e(u,i)e(u',j)|u \neq u' \in U_2, i\neq j \in \{1,2\}\},\end{align*} forms a subgraph of $G$. In the following, for any integer $n \ge 0$, let us denote by $K_{n,n}^\ast$ the graph obtained from the complete bipartite graph $K_{n,n}$ by deleting the edges of a perfect matching, and for two graphs $G_1$ and $G_2$ on disjoint vertex-sets, let us denote by $G_1+G_2$ the graph obtained from $G_1 \cup G_2$ by adding the edges $\{ab|a \in V(G_1), b \in V(G_2)\}$. We can then observe that the graph $H\subseteq G$ defined above is isomorphic to the graph $K_{|U_1|}+K_{|U_2|,|U_2|}^\ast$. 
\medskip
\paragraph{\textbf{Claim 2.} For every pair of integers $n_1, n_2\ge 0$ such that $n_1+n_2\ge t$, the graph $K_{n_1}+K_{n_2,n_2}^\ast$ contains $K_t$ as a minor.}
\begin{proof}
Let us write the vertex-set of $K_{n_1}$ as $\{1,\ldots,n_1\}$, and the vertex-set of $K_{n_2,n_2}^\ast$ as $\{n_1+1,\ldots,n_1+n_2\} \times \{1,2\}$, such that two vertices $(x,i), (y,i) \in V(K_{n_2,n_2}^\ast)$ are adjacent if and only if $x\neq y$ and $i \neq j$. We claim that $K_{n_2,n_2}^\ast$ contains $K_{n_2}$ as a minor. Indeed, this is easy to see if $n_2\le 2$. Now suppose $n_2\ge 3$ and consider the perfect matching $$M=\{(x,1)(x+1,2)|x\in \{n_1+1,\ldots,n_1+n_2-1\}\} \cup \{(n_1+n_2,1)(n_1+1,2)\}$$ in $K_{n_2,n_2}^\ast$. Pause to note that contracting the edges in $M$ transforms $K_{n_2,n_2}^\ast$ into a complete graph $K_{n_2}$. This shows that $K_{n_2,n_2}^\ast$ always contains $K_{n_2}$ as a minor. It then readily follows that $K_{n_1}+K_{n_2,n_2}^\ast$ contains $K_{n_1}+K_{n_2}=K_{n_1+n_2}\supseteq K_t$ as a minor, and this concludes the proof.
\end{proof}
Recalling that we have $|U_1|+|U_2|=|U|\ge t$, Claim~2 now implies that $H$ (and thus $G$) contains $K_t$ as a minor, and this concludes the proof. 
\end{proof}

\section{Odd Hadwiger's conjecture for Kneser and Schrijver graphs}
In this section, we give the self-contained proof of Theorem~\ref{thm:schrijver}. Again, we work with the notion of odd clique-expansions introduced in the previous section.
\begin{proof}[Proof of Theorem~\ref{thm:schrijver}]
Let $n, k$ be positive integers such that $n \ge 2k$. In order to show that the Schrijver graph $S(n,k)$ contains $K_{n-2k+2}$ as an odd-minor, we will construct $n-2k+2$ vertex-disjoint trees $T_1, \ldots, T_{n-2k}, T_{n-2k+1}, T_{n-2k+2}$ contained in $S(n,k)$, and equip each of them with a proper coloring using colors $1$ and $2$ such that any pair of trees is joined by at least one monochromatic edge. The reader is encouraged to look at Figure~\ref{fig:schrijver} for an illustration of the somewhat technical construction.

\begin{figure}[h]
    \centering
    \includegraphics[scale=0.5]{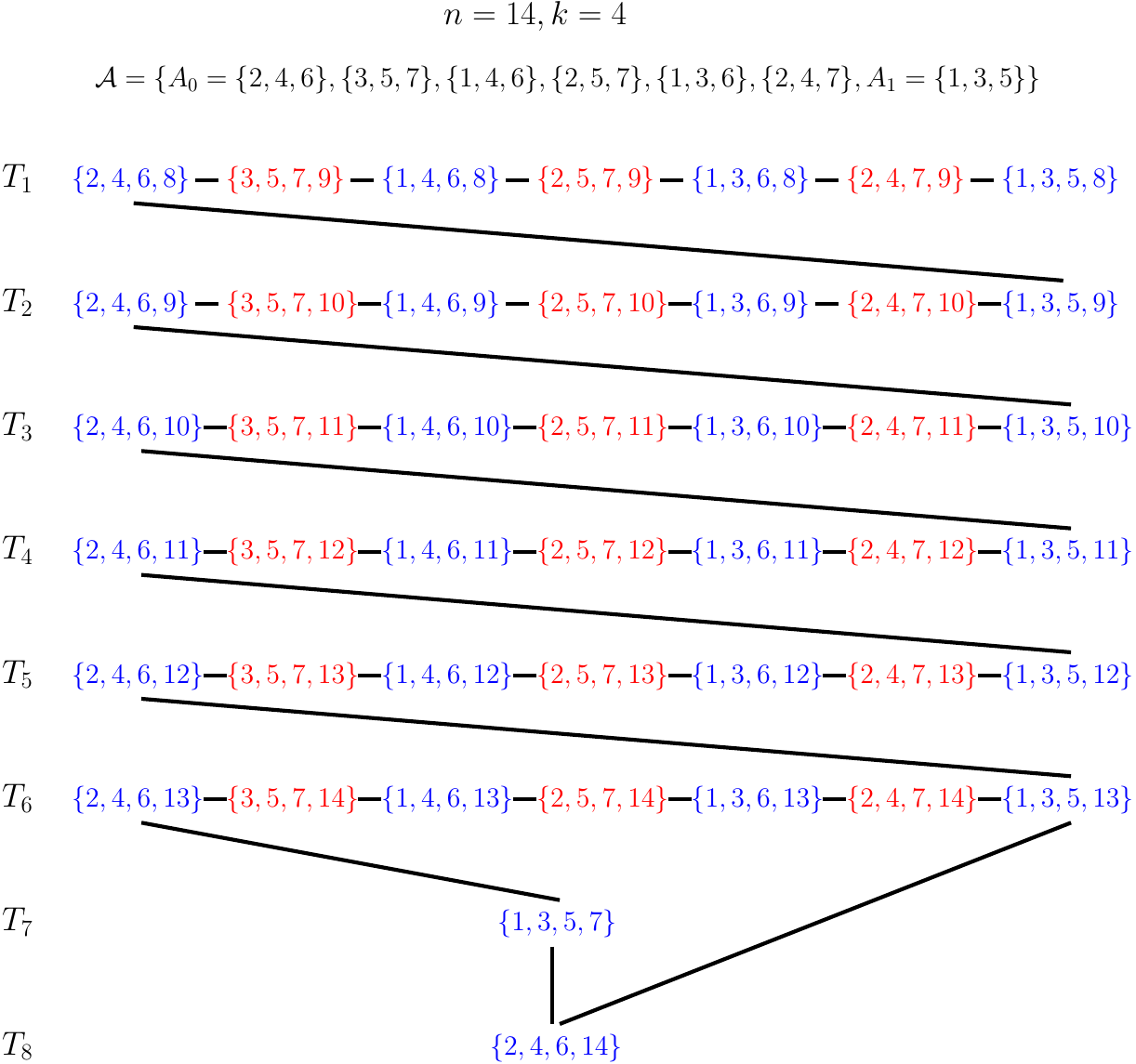}
    \caption{Illustration for the construction of the odd $K_{n-2k+2}$-expansion in $S(n,k)$ for the case $n=14, k=4$. Colors red and blue in the picture correspond to colors $1$ and $2$ in the proof, respectively.  For the sake of readability, only some of the monochromatic edges connecting pairs of distinct trees $T_i$ and $T_j$ are indicated.}\label{fig:schrijver}
\end{figure}

Let us start by describing the first $n-2k$ trees $T_1,\ldots,T_{n-2k}$. We define a collection $\mathcal{A}$ of $(k-1)$-subsets of $[2k-1]$ as follows. A set $A\subseteq [2k-1]$ is contained in $\mathcal{A}$ if and only if it is a vertex of the Schrijver graph $S(2k-1,k-1)$, i.e., if and only if $|A|=k-1$ and $A$ is a cyclically stable subset of $[2k-1]$. In particular this means that no set $A \in \mathcal{A}$ contains both of the elements $1$ and $2k-1$. Moreover, it is easy to see that apart from one special set in $\mathcal{A}$, namely the set $A_0:=\{2,4,\ldots,2k-2\}$ of even elements, all members of $\mathcal{A}$ need to contain \emph{exactly} one of the elements $1$ and $2k-1$. The latter fact allows us to naturally assign colors $1$ or $2$ to each set in $\mathcal{A}$, by defining a mapping $c:\mathcal{A}\rightarrow \{1,2\}$ as $$c(A):=\begin{cases}
    1, & \text{if }A=A_0, \\
    1, & \text{if } A\neq A_0, 1 \in A \text{ and } 2k-1 \notin A, \\
    2, & \text{if } A\neq A_0, 1 \notin A \text{ and } 2k-1 \in A.
\end{cases}$$
For $i=1,\ldots,n-2k$, we now define the vertex-set of the tree $T_i$ as the collection of $k$-sets $\mathcal{A}_i:=\{A\cup \{i+c(A)+(2k-2)\}|A \in \mathcal{A}\}$. Note that each set in $\mathcal{A}_i$ is indeed a cyclically stable $k$-subset of $[n]$: Suppose towards a contradiction that for some $A\in \mathcal{A}$, the set $A\cup \{i+c(A)+(2k-2)\}$ contains two cyclically consecutive elements of $[n]$. Since $A$ is a cyclically stable subset of $[2k-1]$, one of those two elements must be equal to $i+c(A)+(2k-2)$. Since $i+c(A)+(2k-2)\in [2k,n]$ and $A \subseteq [2k-1]$, this is only possible if $i+c(A)+(2k-2)=2k$ and $2k-1 \in A$, or $i+c(A)+(2k-2)=n$ and $1 \in A$. By definition of $c$, we have $c(A)=2$ in the first case and $c(A)=1$ in the second, implying $i=0$ and $i=n-2k+1$ respectively, a contradiction. Thus, we indeed have $\mathcal{A}_i \subseteq V(S(n,k))$ for every $i=1,\ldots,n-2k$, as desired. To finish the description of the trees $T_1,\ldots,T_{n-2k}$, we also have to specify their edge-sets. To do so let us understand the set of edges that $S(n,k)$ induces on the vertex-set $\mathcal{A}_i$. So let $B \cup \{i+c(B)+(2k-2)\}$ and $D \cup \{i+c(D)+(2k-2)\}$ be two members of $\mathcal{A}_i$, where $C, D \in \mathcal{A}$. We then have that $(B \cup \{i+c(B)+(2k-2)\})\cap (D \cup \{i+c(D)+(2k-2)\})=\emptyset$ if and only if  $B\cap D=\emptyset$ and $c(B)\neq c(D)$. From the definition of $c$ we can see that $B\cap D=\emptyset$ enforces $c(B)\neq c(D)$, unless $c(B)=c(D)=1$ and one of $B$ and $D$ is equal to $A_0$. There is a unique member of $\mathcal{A}$ of $c$-value $1$ that is disjoint from $A_0$, namely the set $A_1:=\{1,3,5,\ldots,2k-3\}$. We therefore find that the two sets $B \cup \{i+c(B)+(2k-2)\}$ and $D \cup \{i+c(D)+(2k-2)\}$ are adjacent in $S(n,k)$ if and only if $B$ and $D$ are adjacent in $S(2k-1,k-1)$ and $\{B,D\}\neq \{A_0 \cup \{i+2k-1\},A_1 \cup \{i+2k-1\}\}$. Since the graph $S(2k-1,k-1)$ is isomorphic to the odd cycle $C_{2k-1}$, this implies that $S(n,k)[\mathcal{A}_i]$ is obtained from a $(2k-1)$-cycle by removing a single edge between the pair $\{A_0 \cup \{i+2k-1\},A_1 \cup \{i+2k-1\}\}$. Thus, the graph $S(n,k)[\mathcal{A}_i]$ is a $(2k-1)$-vertex path with endpoints $A_0 \cup \{i+2k-1\}$ and $A_1 \cup \{i+2k-1\}$. We can therefore simply set $T_i:=S(n,k)[\mathcal{A}_i]$ for every $i=1,\ldots,n-2k$, knowing that these graphs indeed form trees (in fact, paths). We can observe that the coloring $c$ of the sets $\mathcal{A}$ naturally induces a $2$-coloring of each of the trees $T_i$, by assigning to every vertex $X\in V(T_1)\cup\cdots \cup V(T_{n-2k})$ the color given by $c(X\cap [2k])$. It is then not hard to see from the definition of $c$ and the trees $T_i$ that this indeed forms a proper coloring of each of the trees $T_1,\ldots,T_{n-2k}$. 

Finally, let us complete the description of the trees by defining $T_{n-2k+1}$ and $T_{n-2k+2}$. Both of these trees consist of a single vertex, namely $V(T_{n-2k+1})=\{1,3,\ldots,2k-1\}$, and $V(T_{n-2k+2})=\{2,4,\ldots,2k-2,n\}$. Their associated proper colorings both consist of assigning color $1$ to their respective vertex. 

We now claim that the collection of trees $T_1,\ldots,T_{n-2k+2}$, with each tree equipped with a proper $\{1,2\}$-coloring as described above, forms an odd $K_{n-2k+2}$-expansion in the Schrijver graph $S(n,k)$. All that is left to verify here is that (A) all the trees are pairwise vertex-disjoint and (B) for every two distinct $i, j \in [n-2k+2]$ there exists at least one monochromatic edge connecting $T_i$ and $T_j$. 

To verify (A), first note that the sets $\{1,3,\ldots,2k-1\}$, $\{2,4,\ldots,2k-2,n\}$ that form the single vertices of $T_{n-2k+1}$ and $T_{n-2k+2}$ are not contained in any of the sets $\mathcal{A}_i, i=1,\ldots,n-2k$. For $\{1,3\ldots,2k-1\}$ this follows directly by definition, since each set in one of the $\mathcal{A}_i$ has at most $k-1$ elements within $[2k]$, while $\{1,3\ldots,2k-1\}$ has $k$. For $\{2,4,\ldots,2k-2,n\}$ this follows since the only way a set of the form $B \cup \{c(B)+i+(2k-2)\}$ for some $B \in \mathcal{A}$ can contain the element $n$ is if $i=n-2k$ and $c(B)=2$, however, the latter implies $2k-1\in B$, in contrast to the set $\{2,4,\ldots,2k-2,n\}$, which does not contain this element. It thus remains to be shown that $\mathcal{A}_i\cap \mathcal{A}_j=\emptyset$ for $1 \le i<j\le n-2k$. Suppose towards a contradiction that $B\cup \{c(B)+i+(2k-2)\}\in \mathcal{A}_j$ for some $B \in \mathcal{A}$. This implies that $c(B)+j+(2k-2)=c(B)+i+(2k-2)$, a contradiction. 

To verify (B), let any pair $\{i,j\} \subseteq [n-2k+2]$ with $i <j$ be given, and let us find a monochromatic edge between $T_i$ and $T_j$. Suppose first $j \le n-2k$. Then we have $A_0 \cup \{i+2k-1\} \in \mathcal{A}_i=V(T_i)$ and $A_1\cup \{j+2k-1\}\in \mathcal{A}_j=V(T_j)$, and clearly $(A_0 \cup \{i+2k-1\}) \cap (A_1\cup \{j+2k-1\})=\emptyset$. Since the proper colorings of $T_i$ and $T_j$ assign color $1$ to both of these sets, we have found an edge of $S(n,k)$ going between $T_i$ and $T_j$ that is monochromatic, as desired. Next, suppose $j \in \{n-2k+1,n-2k+2\}$ and $i\le n-2k$. It is then easy to see that the pair of sets $A_0 \cup \{i+(2k-1)\}, \{1,3,\ldots,2k-1\}$ (if $j=n-2k+1$), and the pair of sets $A_1 \cup \{i+(2k-1)\}, \{2,4,\ldots,2k-2,n\}$ (if $j=n-2k+2$) forms a monochromatic edge between $T_i$ and $T_j$ in this case. Finally, if $i=n-2k+1, j=n-2k+2$ the unique edge between $T_i$ and $T_j$ is clearly monochromatic. Finally, this shows that $S(n,k)$ contains an odd $K_{n-2k+2}$-expansion, and thus contains $K_{n-2k+2}$ as an odd-minor. This concludes the proof. 
\end{proof}

\end{document}